\documentclass[draft]{amsart}
\usepackage{amssymb, enumerate, xspace, graphicx,stmaryrd}
\usepackage[all]{xy}
\SelectTips{cm}{}

\numberwithin{equation}{section}

1

\numberwithin{subsection}{section}

\allowdisplaybreaks[1]

\newtheorem*{namedtheorem}{\theoremname}
\newcommand{\theoremname}{testing}

\newtheorem{theorem}{Theorem}[section]
\newtheorem{proposition}[theorem]{Proposition}
\newtheorem{proposition-definition}[theorem]
{Proposition-Definition}

\newtheorem{lemma}[theorem]{Lemma}

\theoremstyle{definition}
\newtheorem{definition}[theorem]{Definition}

\newtheorem{example}[theorem]{Example}

\newtheorem{remark}[theorem]{Remark}

\theoremstyle{remark}

 \newcommand\cH{\mathcal{H}}
\newcommand\cI{\mathcal{I}} 
 \newcommand\cL{\mathcal{L}}
\newcommand\cM{\mathcal{M}} 
\newcommand\cO{\mathcal{O}}

\newcommand\CC{\mathbb{C}}

 \newcommand\NN{\mathbb{N}}
 \newcommand\PP{\mathbb{P}}

 \newcommand\ZZ{\mathbb{Z}}

\newcommand{\map}{\varphi_{\Lambda_{2n}}}
\newcommand{\ls}{|\Lambda_{2n}|}

\newcommand\arr{\ifinner\to\else\longrightarrow\fi}

\newcommand\arrto{\ifinner\mapsto\else\longmapsto\fi}

\def\displaytimes_#1{\mathrel{\mathop{\times}\limits_{#1}}}

\def\displayotimes_#1{\mathrel{\mathop{\bigotimes}\limits_{#1}}}

\newcommand\pic{\operatorname{Pic}}

\newcommand{\proj}{\operatorname{Proj}}

\newdir{ >}{{}*!/-5pt/@{>}}

\newcommand\doublelong[2]{\mathbin{\xymatrix{{}\ar@<3pt>[r]^{#1}
\ar@<-3pt>[r]_{#2}&}}}

\newlength{\ignora}

\newcommand{\Sec}{\mathrm{Sec}}

\newcommand{\sym}{\operatorname{Sym}}

\newcommand{\lra}{\longrightarrow}

\newcommand{\ra}{\rightarrow}

\begin{document}

\title[Forgetful linear systems and RNCs over $\cM_{0,2n}^{GIT}$]{Forgetful linear systems on the projective space and rational normal curves over $\cM_{0,2n}^{GIT}$}

\author[Bolognesi]{Michele Bolognesi}

\address{Scuola Normale Superiore\\Piazza dei Cavalieri 7\\
56126 Pisa\\Italy\\ and Institut f\"{u}r Mathematik\\Humboldt Universit\"{a}t zu Berlin\\Rudower Chaussee 25\\
10099 Berlin\\Germany}
\email[Bolognesi]{michele.bolognesi@sns.it; bolognes@math.hu-berlin.de}



\begin{abstract}

Let $\cM_{0,n}$ the moduli space of $n$-pointed rational curves. The aim of this note is to give a new geometric construction of $\cM_{0,2n}^{GIT}$, the GIT compactification of $\cM_{0,2n}$, in terms of linear systems on $\PP^{2n-2}$ that contract all the rational normal curves passing by the points of a projective base. These linear systems are somehow a projective analogue of the forgetful maps between the Mumford-Knudsen compactifications $\overline{\cM}_{0,2n+1}$ and $\overline{\cM}_{0,2n}$, but on the other hand they contract some components of the boundary, yielding then a rational map onto $\cM_{0,2n}^{GIT}$. The construction is performed via a study of the so-called $\textit{contraction}$ maps from $\overline{\cM}_{0,2n}$ to $\cM_{0,2n}^{GIT}$ and of the canonical forgetful maps. As a side result we also find a linear system on $\overline{\cM}_{0,2n}$ whose associated map is the contraction map $c_{2n}$.


\end{abstract}

\maketitle

\section{Introduction}

Despite their quite long story, the moduli spaces $\cM_{0,n}$ of pointed rational curves still have a central role in algebraic geometry. In the last two decades the Mumford-Knudsen compactifications $\overline{\cM}_{0,n}$ have drawed the attention of many mathematicians, especially in connection with mathematical physics and enumerative geometry. For instance, their Chow ring was completely described in \cite{keel} and much work has been made to prove the Fulton-Faber conjectures on effective divisors and curves (see e.g. \cite{gabigib}, \cite{keelkernan}, \cite{verme}). These conjectures basically say that the cone of effective curves (respectively effective divisors) is generated by the one-dimensional boundary strata (resp. by the boundary divisors). While the conjecture on curves has been proved, to my knowledge, for $n\leq 7$ \cite{keelkernan}, that about divisors is known to be false \cite{verme}.

\smallskip

Due to their relation with rational normal curves the construction of the spaces $\cM_{0,n}$ can be considered almost classical. They parametrize ordered configurations of distinct $n$ points on the projective line and thus they are not compact, since one expects to have limit conifigurations where two points coincide in some sense.

\smallskip

In order to get a compactification of $\cM_{0,n}$ there exist two main roads. The first is via Geometric Invariant Theory (see for instance \cite{do:pstf}), and it consists first in considering the algebra $R_1^n$ of $PGL(2)$-invariant sections of all the powers of a polarization $L$ on $(\PP^1)^n:=\PP^1\times \dots \times \PP^1$. Then one constructs the GIT compactification $\cM_{0,n}^{GIT}$ as $\proj(R_1^n)$, and this variety contains all the stable and semistable points with respect to the given linearization. Moreover $\cM_{0,n}^{GIT}$ comes with a natural embedding in the projectivized space $\PP (H^0((\PP^1)^n,L)^{PGL(2)})^*$ of invariant sections of the polarization.\\ 

Roughly speaking, via GIT one gets a compact space since this compactification allows two (or more) points to coincide. We recall moreover that if $n$ is even then there exist strictly semistable points and they are the singular locus of $\cM_{0,n}^{GIT}$, whereas if $n$ is odd then stable and semistable points coincide and the moduli space is smooth.

\smallskip

The second main compactification of $\cM_{0,n}$ is the Mumford-Knudsen one, usually denoted $\overline{\cM}_{0,n}$, that is given by stable curves. The points of $\overline{\cM}_{0,n}$ correspond to isomorphism classes of objects of the form $(C,x_1,\dots, x_n)$ where $C$ is a complete curve of arithmetic genus 0 with at most ordinary double points and the $x_i$ are distinct smooth points of $C$. The fact that $\overline{\cM}_{0,n}$ is compact, which may seem counterintuitive is due to the following fact. When one forces two points $x_j,x_k$ to coincide at a smooth point $p\in C$, the limit curve in $\overline{\cM}_{0,n}$ has a new component glued at the point $p$ and $x_j,x_k$ are distinct points on this new component. If $n\geq 3$, $\overline{\cM}_{0,n}$ is a genuine algebraic variety and, unlikely the GIT compactification, it is always smooth.

\smallskip

Both $\overline{\cM}_{0,n}$ and $\cM_{0,n}^{GIT}$ contain $\cM_{0,n}$ as an open subset but the 
compactification given by $\overline{\cM}_{0,n}$ is slighlty finer on the boundary. Moreover there exists a \textit{contraction} surjective morphism 

$$c_n:\overline{\cM}_{0,n} \ra \cM_{0,n}^{GIT}$$

that contracts some subschemes of the boundary of $\overline{\cM}_{0,n}$ while it is an isomorphism on the open set $\cM_{0,n}$.

\smallskip

In this paper we go through the relation between different compactifications of $\cM_{0,n}$, rational normal curves and linear systems on the projective space. The link between rational normal curves and $\cM_{0,n}$ was formalized in modern terms by Kapranov in \cite{kapra} and \cite{kapchowquot} (see also \cite{do:pstf}, Sect. III.2). Following Kapranov, by a \textit{Veronese curve} we will mean a rational normal curve of degree $m$ in $\PP^m$, with $m\geq 2$, i.e. a curve projectively equivalent to $\PP^1$ in its $m^{th}$ Veronese embedding.
We recall that any set of $n+3$ points in general position in $\PP^n$ lie on a unique Veronese curve.

\smallskip

Let now $\cH$ be the Hilbert scheme parametrizing all subschemes of $\PP^{n-2}$. Kapranov realized that there is an isomorphism between the subscheme $V_0(p_1,\dots,p_n) \subset \cH$ of Veronese curves passing by $n$ general points in $\PP^{n-2}$ and $\cM_{0,n}$. Moreover he extended this to an isomorphism 

$$\kappa_n:\overline{\cM}_{0,n}\lra V(p_1,\dots, p_n),$$

where $V(p_1,\dots,p_n)$ is the closure of $V_0$ in $\cH$. $V(p_1,\dots, p_n)$ is the subscheme of $\cH$ given by all degree $n-2$ non-degenerate rational curves. Hence basically by taking the closure of $V_0$ one admits into the picture also reducible curves (see Sect. \ref{kapre} for details); we will call these curves in the boundary \textit{reducible Veronese curves}.

\smallskip

Let $n$ be a positive integer. Let us consider a set $W$ of general points $e_i,\ i=1,\dots, 2n$ in $\PP^{2n-2}$ and the linear system $|\Omega_{2n}|$ of forms of degree $n$ that vanish with multiplicity $n-1$ at the $2n$ points of $W$. Let moreover $\varphi_{\Omega_{2n}}$ be the rational map associated to $|\Omega_{2n}|$. The main result of this paper is the following.

\newcommand{\mapd}{\varphi_{\Omega_{2n}}}

\begin{theorem}

There exists an isomorphism $|\Omega_{2n}|\cong \PP H^0((\PP^1)^{2n},L)^{PGL(2)}$ and the closure of the image of $\varphi_{\Omega_{2n}}$ is  $\cM_{0,2n}^{GIT}$. The closure of the fiber of $\varphi_{\Omega_{2n}}$ over each point $p \in \cM_{0,2n}\subset \cM_{0,2n}^{GIT}$ is the Veronese curve associated to that point via the Kapranov isomorphism $\kappa_n$.

\end{theorem}

Furthermore the rational map $\varphi_{\Omega_{2n}}$ is strictly related to the canonical forgetful maps $f_i:\overline{\cM}_{0,2n+1}\ra \overline{\cM}_{0,2n}$. In fact there exists a blow-down birational morphism $b_{2n+1}:\overline{\cM}_{0,2n+1} \ra \PP^{2n-2}$, whose description is due to Kapranov (and that will be described in detail in Section \ref{bang}). The birationality between $\overline{\cM}_{0,2n+1}$ and $\PP^{2n-2}$ implies that there exists a rational map $\map:\PP^{2n-2}\ra\cM_{0,2n}^{GIT}$ generically of relative dimension one that makes the following diagram commute.

\begin{equation}\label{pari}
\xymatrix{ \overline{\cM}_{0,2n+1} \ar[d]_{b_{2n+1}} \ar[r]^{f_{2n+1}} & \overline{\cM}_{0,2n} \ar[d]^{c_{2n}} \\
\PP^{2n-2} \ar[r]^{\map} & \cM_{0,2n}^{GIT} }
\end{equation} 

The map $\map:\PP^{2n-2}\ra\cM_{0,2n}^{GIT}$ is in fact just the composition of the rational inverse of $b_{2n+1}$ with $c_{2n}\circ f_{2n+1}$. It turns out that our map $\mapd$ is exactly the missing arrow $\map$. 
Of course, diagram \ref{pari} exists also for the odd times pointed rational curves, but giving a result similar to ours in this case seems more difficult. In fact in the proof of Proposition \ref{main}, and notably in the application of Bezout theorem, we make substantial use of the parity of the number of marked points. It is likely however that the odd case will be object of further work.

\smallskip



We remark moreover that taking the closure of the image of the map $\mapd$ gives a completely geometric construction of $\cM_{0,2n}^{GIT}$. Another explicit description of the same moduli spaces by means of linear systems on projective spaces was given by C.Kumar in his beautiful papers \cite{ku:ivb} and \cite{kumar2}. Continuing the work of Coble on the Weddle manifold \cite{cobwed}, \cite{cobl2} Kumar generalized the classical rationalization of the Segre cubic (i.e. $\cM_{0,6}^{GIT}$) given by quadrics passing by 5 general points in $\PP^3$ and gives a vector bundle theoretical interpretation to this construction. As Igor Dolgachev pointed out, Kumar's description and ours are closely related. The relation is outlined in Section \ref{relation}, where we also show that the resolution to $\overline{\cM}_{0,2n}$ of Kumar's maps gives exactly the contraction map $c_{2n}$. 

\smallskip

I was not an expert of this field until, while I was working with A.Alzati on \cite{albol}, we came across these strange linear systems contracting Veronese curves: in fact the case $n=3$ is already described in there. Thus I've been emailing and asking some people who have been helping me very much. Among them I would like to mention Alberto Alzati, Andrea Bruno, Renzo Cavalieri, Igor Dolgachev, Gavril Farkas, Brandon Hasset, Pietro Pirola, Dajano Tossici and especially Angelo Vistoli. Thanks also to Silvia Brannetti for a TeXnical consulence. It is possible that some experts of this field may already know some of the results contained in this paper, but the lack of a precise reference (to my knowledge) and the beauty of the subject pushed my to write them down.

\smallskip

For reasons of simplicity, in this paper we work over the field $\CC$ of complex numbers.

\medskip

\textbf{Description of the contents.}

\medskip

In Section 2 we give a brief account of the Knudsen-Mumford and GIT compactifications.
In Section 3, after introducing the contraction map $c_n:\overline{\cM}_{0,n} \ra \cM_{0,n}^{GIT}$ and the construction of $\overline{\cM}_{0,n}$ as a blow-up of $\PP^{n-3}$, we define the rational map $\map$. Finally in Section 4, after some general remarks on linear systems on the projective space we prove the main theorem by showing that the map $\map$ defined in Section 3 is defined by our linear system $|\Omega_{2n}|$. Moreover, in Section \ref{relation} we go through the relation beween our linear systems and those defined in \cite{kumar2} and give an explicit complete linear series on $\overline{\cM}_{0,2n}$ which defines the contraction morphism $c_{2n}$. Finally we relate our costructions to some recent results (\cite{brumel},\cite{GKM}) on the effective cone of $\overline{\cM}_{g,n}$ and its fibrations.

\section{Stable curves vs. Geometric Invariant Theory}

\subsection{Compactification via stable curves.}\label{kapre}

\begin{definition}

A stable $n$-pointed curve of genus 0 is a connected (possibly reducible) curve $C$ togehter with a set of $n$ ordered smooth distinct points $x_1,\dots,x_n \in C$, that satisfies the following properties:

\begin{enumerate}
\item C has only ordinary double points and every irreducible component of $C$ is isomorphic to the projective line $\PP^1$.
\item The arithmetic genus of $C$ is 0.
\item On each component of the curve there are at least three points that are either marked or double.
\end{enumerate}

\end{definition}

The fine moduli space $\overline{\cM}_{0,n}$ was constructed by Knudsen in \cite{knudsenpointed}, where he also proved that it is a smooth compact algebraic variety for $n\geq 3$. He also showed that there exist $n$ forgetful morphisms

$$f_i: \overline{\cM}_{0,n} \longrightarrow \overline{\cM}_{0,n-1}.$$

The image under $f_i$ of a $n$-curve is obtained by forgetting the $i$ labelled point and passing to the stable model. This morphism makes $\overline{\cM}_{0,n}$ the universal curve over $\overline{\cM}_{0,n-1}$. 

\smallskip

As anticipated in the Introduction, Kapranov in \cite{kapra} gave a very nice description of $\overline{\cM}_{0,n}$ as a closed subvariety $V(p_1,\dots,p_n)$ of the Hilbert scheme $\cH$ parametrizing all subschemes of $\PP^{n-2}$. 
It is useful to understand what are the curves parametrized by the boundary $V(p_1,\dots,p_n)/V_0(p_1,\dots,p_n)$. 
They are reducible Veronese curves, i.e. reducible non-degenerate degree $n$ curves in $\PP^n$ such that each component is a Veronese curve in its projective span. These correspond, via $\kappa_n$, to the curves contained in $ \partial \overline{\cM}_{0,2n}:=\overline{\cM}_{0,2n}/\cM_{0,2n}$. For instance, a 5-pointed stable curve corresponds to a cubic curve $G$ in $\PP^3$ passing by 5 general points. Notably if the stable curve is reducible and has two components, then $G$ is the reducible Veronese curve in $\PP^3$ given by the union of a conic lying on the plane spanned by 3 of the 5 points and a line passing by the other 2 points.

\medskip

\subsection{GIT quotients and their compactifications.}

The GIT compactification of $\cM_{0,n}$ comes from a different framework. Here in fact we are concerned by
the variety $(\PP^1)^n:=\PP^1\times \dots \times \PP^1$ and the diagonal action of $PGL(2)$ on $(\PP^1)^n$

\begin{eqnarray}
\sigma_n: PGL(2) \times (\PP^1)^n & \rightarrow & (\PP^1)^n;\\
\sigma_n(g,(x_1,\dots,x_n)) & \mapsto & (g x_1, \dots, g x_n).
\end{eqnarray}

Let us consider the ample line bundle $\mathcal{L}:=\boxtimes_{i=1}^n \cO_{\PP^1}(1)$ on $(\PP^1)^n$. There exists a unique $PGL(2)$-linearization for $\cL$ (\cite{do:pstf}, Chapter I, Prop. 1) thus the algebra

$$R_1^n:= \bigoplus_{k=0}^{+\infty}H^0((\PP^1)^n,\mathit{L}^k)^{PGL(2)}$$

 of $PGL(2)$ invariant sections of (powers of) $\mathit{L}$ is canonically defined. Since $PGL(2)$ is a reductive algebraic group, $R_1^n$ is of finite type over $\CC$, thus we define $\cM_{0,n}^{GIT}$ as the projective algebraic variety $\proj(R_1^n)$, that is naturally embedded in $\PP (H^0((\PP^1)^n,L)^{PGL(2)})^*$. Moreover we remark that $(\PP^1)^n$ is a proper algebraic variety, and the action is regular. This means that we can consider the set $(\PP^1)^n_{ss}(\cL)$ of semistable points of  $(\PP^1)^n$ with respect to the linearized invertible sheaf $\cL$. Now, its categorical quotient $(\PP^1)^n_{ss}(\cL)/PGL(2)$ exists \cite{GIT} and we have an isomorphism

$$(\PP^1)^n_{ss}(\cL) \cong \proj(R_1^n).$$

Let $(\PP^1)^n_{s}(\cL)$ be the subset of stable points. The quotient $(\PP^1)^n_{s}(\cL)/PGL(2)$ is
a geometric quotient for $(\PP^1)^n_{s}(\cL)$ and, if $n$ is even, it is a proper open subset of $\proj(R_1^n)$. In this case in fact the locus of strictly semistable points is the singular locus of $\cM_{0,n}^{GIT}$, which is made up by $\left(\begin{array}{c}
n \\ \frac{n}{2}\\
\end{array}\right)/2$ ordinary double points.
On the other hand, when $n$ is odd $(\PP^1)^n_{ss}(\cL)=(\PP^1)^n_{s}(\cL)$ and $\cM_{0,n}^{GIT}$ is smooth.

\medskip

We remark that, when $n\leq 5$ the two compactifications described in this Section give rise to the same algebric varieties, whereas already for $n=6$, $\overline{\cM}_{0,6}$ is a small resolution of $\cM_{0,6}^{GIT}$, which is the Segre cubic.

\smallskip

In order to avoid confusion, in the following we will often call \textit{a stable curve} an element of $\overline{\cM}_{0,n}$ whereas by \textit{a configuration of points} we will mean an element of $\cM_{0,n}^{GIT}$.

\section{A map between moduli spaces $\overline{\cM}_{0,2n+1}\ra \cM_{0,2n}^{GIT}$.}

The goal of this section is to introduce in a moduli theoretical way a dominant rational map to $\cM_{0,2n}^{GIT}$ 

$$\varphi_{\Lambda_{2n}}: \PP^{2n-2} \longrightarrow \cM_{0,2n}^{GIT}\subset \PP (H^0((\PP^1)^{2n},L)^{PGL(2)})^*$$

\subsection{The contraction map $c_n: \overline{\cM}_{0,n} \ra \cM_{0,n}^{GIT}$.}\label{moduli}

In order to define the map $\map$ we need to introduce the contraction map $c_n:\overline{\cM}_{0,n} \ra \cM_{0,n}^{GIT}$. As it is well known the Mumford-Knudsen compactification provide way to compactify $\cM_{0,n}$ which is finer than that given by the GIT one. Both contain $\cM_{0,n}$ as an open set but on the boundary $\overline{\cM_{0,n}}$ reveals its finer nature. For instance consider an even $n > 3$, say 6, and the locus of curves like that in Figure \ref{pict}.

\begin{figure}[h!]
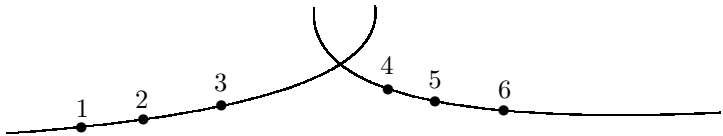

 \centering
 \input figura1.pic\caption{\small{A stable 6-pointed curve that determines a strictly semistable configuration of points.}}\label{pict}
 \end{figure}

These curves have two rational components, and the first $\frac{n}{2}=3$ points lie on one component and the last 3 on the other. The locus parametrizing these type of curves in $\overline{\cM_{0,6}}$ is a divisor whereas the locus of the associated configuration of points in $\cM_{0,6}^{GIT}$ (obtained by contracting one of the components of the curve) is just one strictly semistable point. In a very interesting and useful way, this processus of contracting components can be made "universal" and described by a \textit{contraction} morphism $c_{n}:\overline{\cM}_{0,n} \ra \cM_{0,n}^{GIT}$ which is the identity on $\cM_{0,n}$ and that is dominant on the boundary of $\cM_{0,n}^{GIT}$. The mere existence of the map follows from \cite{kapchowquot} where it is said that the moduli space of stable $n$-pointed curves is the inverse image of all GIT quotients of $(\PP^1)^n$ by all possible linearizations. An explicit description of this map has been given in \cite{lanar} where the following theorem is proven.

\begin{theorem}\cite{lanar}
The canonical isomorphism between the open susbsets of $\cM_{0,n}^{GIT}$ and $\overline{\cM_{0,n}}$ isomorphic to $\cM_{0,n}$ extends to a regular morphism

$$c_{n}:\overline{\cM}_{0,n} \ra \cM_{0,n}^{GIT}.$$

\end{theorem}

The precise statement in \cite{lanar} regards marked, and not pointed, curves, i.e. it does not care about the order of the special points, but the proof extends without any harm to the framework of pointed curves so we won't go too much into the details here. The interested reader is suggested to look in \cite{lanar} where this and other related results are described in great detail. Anyway for sake of completeness we will describe how $c_n$ is defined on the boundary.

\smallskip

The key idea of the proof in \cite{lanar} is to show that for any stable $n$-marked rational curve there exists a unique \textit{central component}, on which one can contract all the other branches in order to obtain a stable configuration of points. This is seen in a clearer way by considering the dual graph of the curve.

\smallskip

Recall in fact that to any genus 0 nodal curve $C$ we can associate its dual graph. In this context, each vertex
of the graph represents a component of $C$ and two vertices are connected
by an edge if the corresponding components intersect. Hence, to every stable nodal curve of genus zero we associate a connected tree whose vertices are in bijection with the irreducible rational components of the curve itself. 
To an $n$-pointed nodal curve $(C, p_1,\dots , p_n)$ of genus 0 we associate a weighted tree in the following way. Let $v_i$ be the vertex corresponding to the component $C_i$; we associate to $v_i$ the weight $w_i$ which is the number of marked points lying on $C_i$. The total weight $W(T)$ of a tree is simply the number of marked points of the tree $T$.
If $v\in T$ is a vertex, we denote by $e_j(v)$ the edges with one end in v and call them
the edges starting at v. A connected weighted tree is called stable if for
every vertex $v_i$ we have $w_i + \sum_{j} e_j(v_i) \geq 3$. It is easy to see that a pointed stable curve of 
genus 0 is stable if and only if its corresponding tree $T$ is stable.

Let $v$ be a vertex of an $n$-weighted tree $T$ and $e_1(v),\dots , e_n(v)$
the edges of $T$ starting at $v$. Then the subgraph $T \ \{v, e_1(v),\dots , e_n(v)\}$ consists
of $n$ weighted trees $T_1, \dots , T_n$ complementary
to the vertex $v$. A vertex $v$ of an $n$-weighted tree $T$ is called \textit{a central vertex}
if $W(Ti) < \frac{m}{2}$ for every subtree complementary to $v$.

\smallskip

It can be checked that, under certain assumptions, an $n$-weighted tree always has a central vertex.
This results is translated in the language of stable curves in the following terms.

\begin{lemma}
Given a stable $n$-pointed curve $(C,p_1,\dots,p_n)$, there exists always a central component, unless $n$ is even and the curve has two components containing $\frac{n}{2}$ points each.
\end{lemma}

When such a central component exists then the map $c_{n}:\overline{\cM}_{0,n} \ra \cM_{0,n}^{GIT}$ consists in contracting the non-central components of a stable curve on the central one thus obtaining a (semi-) stable configuration of points. When the central component does not exist one fixes the problem by hand since the 
$\left(\begin{array}{c}
n \\ \frac{n}{2}\\
\end{array}\right)$ boundary divisors in $\cM_{0,n}$ that paramtrize curves without central components are contracted, two by two in the obvious way, on the 
$\left(\begin{array}{c}
n \\ \frac{n}{2}\\
\end{array}\right)/2$ semistable points and it is easy to see that the obtained map is regular everywhere.

\medskip

\subsection{The blow-up construction of $\overline{\cM}_{0,n}$.}\label{bang}

Now we recall a construction of $\overline{\cM}_{0,n}$ as a sequence of blow-ups of $\PP^{n-3}$  that was first given by Kapranov. The way it is presented here owes anyway a lot to \cite{hassetpointed}.

\begin{theorem}\label{bloua}(\cite{hassetpointed} Sect. 6.2)

The Mumford-Knudsen compactification $\overline{\cM}_{0,n}$ has the following realization as a sequence of blow-ups of $\PP^{n-3}$. Let $q_1,\dots,q_{n-1}$ be general points in $\PP^{n-3}$:\\

\smallskip

\noindent 1: blow up the points $q_1,\dots, q_{n-1}$;\\
2: blow up proper transforms of lines spanned by pairs of the points
$q_1,\dots, q_{n-1}$;\\
3: blow up proper transforms of 2-planes spanned by triples of the points $q_i$; . . .\\
n-4: blow up proper transforms of (n-5)-planes spanned by (n-4)-tuples
of the points $q_i$.
\end{theorem}

The corresponding blow-down map $b_n:\overline{\cM}_{0,n}\ra \PP^{n-3}$ has the following property (see also \cite{keelkernan} Prop. 3.1): the images via $b_n$ of the fibers of the forgetful map $f_n$ over points of ${\cM}_{0,n-1}\subset \overline{\cM}_{0,n-1}$ are the rational normal curves in $\PP^{n-3}$ passing by the $n-1$ general points.

\medskip

Before we define our map $\varphi_{\Lambda_{2n}}: \PP^{2n-2} \longrightarrow \cM_{0,2n}^{GIT}$ we shall introduce a birational isomorphism $r_{2n+1}:\PP^{2n-2}\longrightarrow \overline{\cM}_{0,2n+1}$.

In order to do this we need to recall from \cite{kapra} the Kapranov isomorphism $\kappa_n$: the moduli space $\overline{\cM}_{0,n}$ of stable curves is isomorphic to the variety $V(p_1,\dots,p_n)\subset \cH$ that parametrize (possibly degenerate) Veronese curves passing by the general points $p_i$ in $\PP^{n-2}.$ A trivial example of this construction is $\overline{\cM}_{0,4}\cong \PP^1 \cong |\cI_{p_i}(2)|$, where the points $p_i$ are four general points in $\PP^2$.\\

Our rational map $r_{2n+1}$ is basically an inverse of the blow-down map $b_{2n+1}$. Suppose in fact that we have fixed $2n$ general points $e_i$ in $\PP^{2n-2}$ (a projective base). Then a general point $p\in\PP^{2n-2}$ defines a unique Veronese curve, notably that passing by the $2n+1$ points given by the $e_i$ plus $p$. This, via the isomorphism $\kappa_{2n}$, corresponds to a pointed curve in $\overline{\cM}_{0,2n}$, to which we add the $(2n+1)^{th}$ marking given by $p$ itself. Thus this yields a stable $(2n+1)$-pointed curve which is clearly the inverse image of $p$ via 
$b_{2n+1}$. The following proposition formalizes this.

\begin{proposition}

The rational map $r_{2n+1}:\PP^{2n-2} \longrightarrow \overline{\cM}_{0,2n+1}$ just defined is a rational inverse of the blow-down morphism $b_{2n+1}$.
 
\end{proposition}

We remark that the rational map $r_{2n+1}$ here above is just a modified, rational version of the GIT association isomorphism 

\begin{equation}\label{asso}
(\PP^{1})^m/PGL(2) \cong (\PP^{m-3})^m/PGL(m-3)
\end{equation}

given by associated point sets in the projective spaces \cite{do:pstf} (Chap III, Prop. 2), \cite{coble}.

\begin{definition}

We call $\varphi_{\Lambda_{2n}}$ the rational map that is the composition of $r_{2n+1}$  with $c_{2n} \circ f_{2n+1}$, so that the follwing diagram commutes.

\begin{equation}\label{diagram}
\xymatrix{ \overline{\cM}_{0,2n+1} \ar[d]_{b_{2n+1}} \ar[r]^{f_{2n+1}} & \overline{\cM}_{0,2n} \ar[d]^{c_{2n}} \\
\PP^{2n-2} \ar[r]^{\map} & \cM_{0,2n}^{GIT} }
\end{equation}

\end{definition}

The linear system that defines $\map$ will be denoted $|\Lambda_{2n}|$.

\section{Linear systems on $\PP^n$ associated to the forgetful maps.}\label{kumm}

It is well known that $\pic(\PP^n)=\ZZ=<\cO_{\PP^n}(1)>$ hence every linear system on $\PP^n$ is a non-trivial vector subspace of $H^0(\PP^n,\cO_{\PP^n}(r))\cong\sym^r(\CC^{n+1})^*$ for some $r>0$. In the following we will denote  $[x_0,\dots,x_n]$ the homogeneus coordinates on $\PP^n$. 

\begin{definition} \label{compo}

We  say that a degree $r$ form $Q$ vanishes at a point $p\in\PP^n$ with multiplicity $y>0$ if all the partial derivatives of order $y-1$ of $Q$ vanish at $p$. 

\end{definition}

We will say that an $r$-form vanishes on a subvariety $V$ of $\PP^n$ with multiplicity $y$ if it vanishes at each point of $V$ with multiplicity at least $y$.

We recall the following easy remarks:

\begin{remark}\label{forms}

\begin{enumerate}

\item If an $r$-form $Q$ on $\PP^n$ vanishes on two points $p$ and $q$ with multiplicities $y_p$ and $y_q$ respectively, then $Q$ vanishes on the line $<p,q>$ spanned by the two points with multiplicity at least $y_p+y_q-r$.

\item If an $r$-form $Q$ on $\PP^n$ vanishes at a finite family of points $\{q_i\}$ in general position with multiplicities $r-1$, then $Q$ vanishes on all the $(k-1)$-dimensional linear subvarieties $<q_{i_1}, \dots , q_{i_k}>$ of $\PP^n$ with multiplicities $r-k$ for $k<r$. This means that the restriction of $Q$ to an $(r-1)$-plane $<q_{i_1}, \dots, q_{i_r}>$ is a scalar multiple of the product of hyperplanes $<q_{i_1}, \dots, \hat{q}_{i_j}, \dots, q_{i_r}>$.

\item Let $e_j = [0 : \dots : 1 : \dots : 0] \in \PP^n$ ($e_j = 1$ at the $j^th$ place). Then an $r$-form
$Q$ vanishes at points $e_j$ (with $j = 1, . . . , n+1$) with multiplicities $r-1$ if and only if it is a
linear combination of monomials $x_1^{\alpha_1}x_2^{\alpha_2}\dots x_{n+1}^{\alpha_{n+1}}$  with $\alpha_j = 0$ or 1 and $\sum_{k=1}^{n+1} \alpha_k=r$. Generally speaking, an $r$-form $Q$ vanishes at points $e_j$ (with $j = 1, . . . , n+1$) with multiplicities $r-m_j$ if and only if it is a linear combination of monomials $x_1^{\alpha_1}x_2^{\alpha_2}\dots x_{n+1}^{\alpha_{n+1}}$ with
$0 \leq \alpha_j \leq m_j$ and
$\sum^{n+1}_{k=1} \alpha_k = r$.

\end{enumerate}

\end{remark}

Let $W$ be the set of the $2n$ fixed general points $e_i$, $i=1,\dots, 2n$, in $\PP^{2n-2}$. For simplicity we will always consider $W$ as a projective base, since any set of this kind can be transformed, via a projective automorphism of $\PP^{2n-2}$, into the set given by the $2n-1$ coordinate points $[0:\dots:1:\dots:0]$ and $[1:\dots:1]$. The most important consequence that we draw out of the prededing remarks is that the linear system of forms of degree $n$ that vanish on the points $e_i$ with multiplicity $n-1$ also vanish on the $(n-2)$-linear spans of $(n-1)$-ples of points of $W$.

\smallskip


Now we recall (Definition \ref{compo}) that our rational map $\map:\PP^{2n-2} \longrightarrow \cM_{0,2n}^{GIT}$ is exactly the composed map $c_{2n} \circ f_{2n+1} \circ r_{2n+1}$. Hence the linear system on $\PP^{2n-2}$ associated to $\map$ has the same dimension of $\PP H^0((\PP^1)^{2n},L)^{PGL(2)}$ and its divisors are the pull-back via $c_{2n} \circ f_{2n+1} \circ r_{2n+1}$ of those of $\cO_{\cM_{0,2n}^{GIT}}(1)$.


\smallskip

Let us denote $\ls$ the linear system associated to $\map$. In order to understand what linear system is $\ls$, we make a few easy remarks on the behaviour of $\map$ on certain linear subspaces spanned by subsets of $W$. We recall moreover that the proper transforms in $\overline{\cM}_{0,2n+1}$ (seen as a blow-up of $\PP^{2n-2}$) of linear spans of subsets of $W$ of cardinality at most $2n-3$ give all the boundary divisors of the moduli space; see \cite{keelkernan} Sect. 3 for a good description of this. One recovers the fact that the strict transform of these linear spans parametrize reducible curves from the following easy fact. All Veronese curves passing by the $2n$ points $e_i$ and a point $p$ belonging to a linear span of (at most) $2n-3$ points belonging to $W$ are $2n+1$-pointed reducible Veronese curves. More precisely the general curve belonging to these boundary divisors has two components.

\begin{theorem}\label{base}
The base locus of the linear system $\ls$ is (set-theoretically) the configuration of $(n-2)$-dimensional linear spaces spanned by all the $(n-1)$-ples of points contained in the set $W$. All the $(n-1)$-dimensional linear spans of $n$-ples of points in $W$ are contracted to points. In particular the two $(n-1)$-planes spanned by complementary $n$-ples of points in $W$ are contracted to the same point.
\end{theorem}

\begin{proof}

We remark that if $p$ is a point contained in the span of any $(2n-3)$-plet contained in $W$ then there exists no unique degenerate Veronese curve $C_1$ passing by the set $W$ plus $p$. This means that $r_{2n+1}$ is not defined in the locus given by this configurations of $\PP^{2n-4}$. This is basically part of the argument hidden behind the blow-up construction of $\overline{\cM}_{0,2n+1}$ of Thm. \ref{bloua}.\\
Consider now the points belonging to the linear spans of $s$-plets of points of $W$ with $s\geq n$ that are not contained in any linear span of $t$-plets of points with $t<n$. This is a quasi-projective sub-variety of $\PP^{2n-2}$ that we will denote $Y$. We can suppose that $z$ is general in the respective $(s-1)$-dimensional linear span, otherwise one can reduce to a lower dimensional linear envelop of points of $W$. The fiber of $b_{2n+1}$ over any point $z\in Y$ represents all the (reducible) Veronese curves passing by $W$ and $z$. Let us now consider these Veronese curves.\\
Suppose first that $s>n$. Then we see that any of the reducible Veronese curves represented by points of $b_{2n+1}^{-1}(z)$ contains a fixed component $C_1$: $C_1$ is the only rational normal curve in $\PP^{s-1}$ passing by the $s$ points, $z$ and the point $q$, intersection of $\PP^{s-1}$ with the linear span of the other $2n-s$  points of $W$. Thus the fiber of $b_{2n+1}$ over $z$ can be seen to represent the family of Veronese curves passing by $q$ and the remaining $2n-s$ points contained in the linear span of these points. Now recall the definition of the morphism $f_{2n+1}$ and in particular of the contraction map $c_{2n}$. We remark that the fiber $b_{2n+1}^{-1}(z)$ is completely contracted by $c_{2n}\circ f_{2n+1}$ to a unique configuration of points in $\cM_{0,2n}^{GIT}$ where all the component(s) containing the $2n-s$ points of $W$ are contracted on $C_1$ to the point $q$.\\ 
Less intricately, if we assume $s=n$ all the fiber of $b_{2n+1}$ over $z$ is contracted to the semi-stable point of $\cM_{0,2n}^{GIT}$ corresponding to the choice of the $s$ points in $W$ and complementary choices in $W$ of the $s$ points correspond to the same configuration. This implies that the map given by the linear system $\ls$ is well defined and regular on any point of $Y$. A similar argument shows that if $s<n$ then the fiber of $b_{2n+1}$ over $z$ is not completely contracted by $c_{2n}\circ f_{2n+1}$, hence these linear spans are contained in the exceptional locus of the rational map. 
\end{proof}

In the following we will denote $\Sec^k(W)$ the configuration of $k$-dimensional linear subspaces spanned by $(k+1)$ points belonging to $W$. For $k=1$ we will simply write $\Sec(W)$. With this notation the base locus of theorem \ref{base} is denoted $\Sec^{n-2}(W)$.

\begin{remark}\label{bricks}
We recall that any irreducible Veronese curve $B$ passing by the $2n$ points of $\PP^{2n-2}$ is pulled back by $b_{2n+1}$ to a curve in $\overline{\cM}_{0,2n+1}$ that is the fiber of $f_{2n+1}$ over the point of $\overline{\cM}_{0,2n}$ that corresponds to $B$ via the Kapranov bijection $\kappa_n$. This in turn implies that, since by definition diagram \ref{diagram} commutes, the map $\map$ contracts every such rational normal curve.

Moreover, from Thm. \ref{base} we know that each $\PP^{n-1}$ spanned by $n$ of the $2n$ fixed points of $W$ is contracted to a point. This implies the following easy fact. Let $D$ be any divisor of the linear system $\ls$ and $\overline{\PP}^{n-1}$ any of the preceding $(n-1)$-dimensional linear spans; then set-theoretically $D\cap \overline{\PP}^{n-1}= \Sec^{n-2}(W)\cap \overline{\PP}^{n-1}$.
\end{remark}

We recall from the introduction that we denote $|\Omega_{2n}|$ the linear system of forms of degree $n$ that vanish with multiplicity $n-1$ at the points of $W$ and by $\mapd$ the associated rational map.

\begin{proposition}\label{main}

We have an isomorphism of linear systems on $\PP^{2n-2}$

\begin{equation}\label{iso}
\ls \cong |\Omega_{2n}|
\end{equation}

that induces an identity of rational maps 

$$\map=\mapd : \PP^{2n-2} \longrightarrow \cM_{0,2n}^{GIT}\subset \ls^* \cong |\Omega_{2n}|^*.$$

\end{proposition}

\begin{proof}

We recall from Theorem \ref{base} that the base locus of $\ls$ is (set-theoretically) $\Sec^{n-2}(W)$ and from Remark \ref{bricks} that $\map$  contracts each rational normal curve passing by $W$ to a point of the image. Remark furthermore that the intersection of any such rational normal curve with the base locus $\Sec^{n-1}(W)$ is set-theoretically the set $W$ itself.

This means the following. Let $D$ be any divisor of $\ls$, i.e. a homogeneus form of some degree $r$ that vanishes 
on $\Sec^{n-1}(W)$, and  $\mu_1,\dots,\mu_{2n}\in \NN$ the multiplicities of $D$ at the points $e_1,\dots,e_{2n}$ of $W$. Let $\tilde{C}$ be any rational normal curve in $\PP^{2n-2}$ passing by $W$. Then, since set-theoretically $D\cap \tilde{C}=W$ and any rational normal curve is contracted by $\ls$, we have the following equation between integers

\begin{equation}\label{mult}
r(2n-2)=\mu_1+\dots + \mu_{2n}.
\end{equation}

Let $\Sigma_n$ be the symmetric group on $n$ elements. The points $e_i$ of $W$ are labelled with numbers from 1 to $2n$. Hence by the classical properties of projective automorphisms (see for example \cite{har:ag}), for every $\sigma \in \Sigma_{2n}$ there exists a unique automorphism $A_{\sigma}$ of $\PP^{2n-2}$ such that the permutation induced on the points $e_i$ is exactly that corresponding to $\sigma$. 

\smallskip

Moreover it is easy to see that the maps $b_{2n+1}, f_{2n+1}, c_{2n}$ are $\Sigma_{2n}$-equivariant, where the action on $\overline{\cM}_{0,2n}$ and $\cM_{0,2n}^{GIT}$ is the natural one, that on $\PP^{2n-2}$ is given by the automorphisms $A_{\sigma}$ and that on $\overline{\cM}_{0,2n+1}$ is obtained from the natural $\Sigma_{2n+1}$-action by forgetting the last point. Hence, since diagram \ref{diagram} must commute, $\map$ is a $\Sigma_{2n}$-equivariant map, with respect to the natural action on $\cM_{0,2n}^{GIT}$ and via the automorphisms $A_{\sigma}$ on $\PP^{2n-2}$. This implies that the linear system $\ls$ is invariant with respect to the $\Sigma_{2n}$-action induced on $n$-forms by the one on $\PP^{2n-2}$ given by the automorphisms $A_{\sigma}$.

Since $\ls$ is $\Sigma_{2n}$-invariant then $\mu_1=\dots = \mu_{2n}:=\mu$ and, up to multiplication by two, Equation \ref{mult} becomes the integral equation 

\begin{equation}\label{equa}
r(n-1)=n\mu;
\end{equation}

 i.e. the elements of $\ls$ are forms of degree $r$ that vanish with multiplicity $\mu$ at the points $e_i$ for $\mu,n,r \in \NN$ that realize the equality above. These observations imply that $\ls$ can only be a linear subsystem of the linear system of $pn$-forms that vanish with multiplicity $\mu=p(n-1)$ on each point $e_i$, $i=1,\dots,2n$, for any $p\in \NN$. Remark that these forms, by Remark \ref{bricks}, vanish $p$ times on the configuration of planes $\Sec^{n-2}(W)$.

Now we recall that, since diagram \ref{diagram} commutes, the rational map $\map$ defined by $\ls$ should resolve to the morphism $c_{2n}\circ f_{2n+1}$ after the subsequent blow-ups described in Theorem \ref{bloua}. Hence, since if $p>1$ the forms of $\ls$ vanish with multiplicity greater than 1 on $\Sec^{n-1}(W)$, one sees that the only value of $p$ for which such a blow-up resolves the map $\map$ is $p=1$. If $p>2$ in fact one would need further blow-ups. This means that $\ls$ is a linear susbsystem of $|\Omega_{2n}|$. This can be rephrased as following. Let $V_{\Omega_{2n}}$ denote the linear space of forms associated to $|\Omega_{2n}|$ and let $\rho$ be the pull-back map defined by $\map$, then $\rho$ is a non-zero homomorphism 

$$\rho: H^0(\cM_{0,2n}^{GIT},\cO(1))=H^0((\PP^1)^{2n},L)^{PGL(2)} \lra  V_{\Omega_{2n}}.$$

The target space is a $\Sigma_{2n}$-module via the action induced by the projective automorphisms $A_{\sigma}$, $H^0((\PP^1)^{2n},L)^{PGL(2)}$ is an irreducible $\Sigma_{2n}$-module and it can be showed that $\rho$ is equivariant with respect to these $\Sigma_{2n}$-actions.\\ Since $H^0((\PP^1)^n,L)^{PGL(2)}$ is irreducible, $\rho$ must be injective. By checking that $dim(V_{\Omega_{2n}})=dim(H^0((\PP^1)^n,L)^{PGL(2)})=dim(\ls)+1$ we will prove that $\rho$ is in fact an isomorphism and conclude our proof. This computation takes inspiration from those in \cite{kumar2}, Sect. 3.1.

Via the bijection between $\Sigma_n$-modules and Young tableaux, the $\Sigma_{2n}$-module $H^0(\cM_{0,2n}^{GIT},\cO(1))$ corresponds to the Young tableau consisting of 2-rows and $n$-columns, hence by the hook lenght formula \cite{do:pstf} we get

\begin{equation}\label{hook}
dim (H^0(\cM_{0,2n}^{GIT},\cO(1)))=\frac{(2n)!}{(n+1)!n!}.
\end{equation}


Let us now compute the dimension of $V_{\Omega_{2n}}$. Let $N_k$ be the set $\{1,\dots,k\}$ of the first $k$ positive integers. Let us moreover define two other sets:

\begin{eqnarray*}
R:=\{I\subset N_{2n-1} : \#(I)=n\},\\
S:=\{J\subset N_{2n-1}: \#(J)=n-2\}.
\end{eqnarray*}

If $I=\{i_1,\dots,i_r\}$, by $x_I$ we will mean the monomial $x_{i_1}x_{i_2}\dots x_{i_r}$.
By the third observation of Remark \ref{forms} any $n$-form $Q$ vanishing with multiplicity $n-1$ at the $2n-1$ coordinate points $e_1,\dots,e_{2n-1}$ is of the type $\sum_{I\in R} a_I x_I$ with $a_I\in \CC$. Then we have to add the condition of vanishing with multiplicity $r-1$ at $e_{2n}$ thus obtaining

$$dim(V_{\Omega_{2n}})=\bigl\{ \sum_{I\in R} a_Ix_I: \sum_{J\subset I \in R} a_I=0, \mathrm{for\ all\ } J\in S \bigr\}.$$

Let $(\lambda_{IJ})_{I\in R, J\in S}$ be the incidence matrix defined by the condition $\lambda_{IJ}$ if $J\subset I$ and $\lambda_{IJ}=0$ if $J\not \subset I$. The matrix $(\lambda_{IJ})$ has maximal rank hence the conditions among the generators $\{x_I: I\in R\}$ of $V_{\Omega_{2n}}$ are indipendent. This implies that 

\begin{equation}\label{dimens}
dim(V_{\Omega_{2n}})= \#(R) - \#(S)= \left(\begin{array}{c}
2n-1 \\ n\\
\end{array}\right) - \left(\begin{array}{c}
2n-1 \\ n-2 \\
\end{array}\right).
\end{equation}

It is easy then to see that Eq. \ref{dimens} is equal to $\frac{(2n)!}{(n+1)!n!}$, which was the correct value given by the Hook formula in Equation \ref{hook}. This implies that $\rho$ is an isomorphism and concludes the proof.

\end{proof}

\section{The relation with Kumar's linear systems and with the effective cone of $\overline{\cM}_{0,n}$}\label{relation}

\subsection{The relation with Kumar's work}

In this section we will show the relation between our linear systems $|\Omega_{2n}|$ and the ones considered in \cite{kumar2}; this part dues a lot to a remark by Igor Dolgachev. Let us first recall the following classical definition, that we will be using in the following.

\begin{definition}
A Cremona inversion is a birational involution of the projective space defined as follows 

\begin{eqnarray*}
Cr_n: \PP^n & \longrightarrow & \PP^n; \\
\left[ x_0:\dots:x_n \right] & \mapsto &  \left[1/x_0:\dots : 1/x_n \right].
\end{eqnarray*}

\end{definition}

Given the usual projective base $e_1,\dots,e_{n+2}$ for $\PP^n$ as in the previuos sections, the Cremona inversion has the following property. If $l$ is a line in $\PP^n$ passing by $e_{n+2}$ and not lying in the span of any subset of the other coordinate points, then $Cr_n(l)$ is a rational normal curve passing by all the $n+2$ points of the projective base. One can define Cremona inversions for any projective base and the points $e_1,\dots, e_{n+1}$ are called the \textit{base} of the Cremona inversion.

\begin{example}
On $\PP^2$ the Cremona inversion is given by the linear system of quadrics passing by 3 general points. On $\PP^3$ by the cubics that vanish on the six secant lines of 4 general points, and so on.
\end{example}

Because of their relation with rational normal curves Cremona inversions are, up to a projection $\pi_{n-1}:\PP^{n-3} \ra \PP^{n-4}$ with center the last point $e_{n-1}$, the projective analogue of the forgetful map $f_{n}:\overline{\cM}_{0,n} \ra \overline{\cM}_{0,n-1}$. By this we mean that the following diagram is commutative (see \cite{kapra}, Prop. 2.10).

\begin{equation*}
\xymatrix{ \overline{\cM}_{0,n} \ar[d]_{b_{n}} \ar[rr]^{f_{n}} & & \overline{\cM}_{0,n-1} \ar[d]^{b_{n-1}} \\
\PP^{n-3} \ar[rr]^{\pi_{n-1}\circ Cr_{n-3}} & & \PP^{n-4} 
  }
\end{equation*} 

Let us now come back to our linear systems and give a brief account of Kumar's work. Let $|\Xi_{2n}|$ be the linear system on $\PP^{2n-3}$ given by degree $n-1$ hypersurfaces vanishing at $2n-1$ fixed points in general position with multiplicity $n-2$. In \cite{kumar2} the author shows that there exists an isomorphism between $\PP (H^0((\PP^1)^n,L)^{PGL(2)})^*$ and the linear system $|\Xi_{2n}|^*$ and that $\PP^{2n-3}$ is mapped birationally by $|\Xi_{2n}|$ onto $\cM_{0,2n}^{GIT}$. His construction relies on the association isomorphism (see Equation \ref{asso}) and moduli theoretically the map given by $|\Xi_{2n}|$ works as follows (see also \cite{do:pstf}, Chapter III for more details). We set $2n-1$ fixed points in $\PP^{2n-3}$, hence for every general point $p\in \PP^{2n-3}$ there exists a unique rational normal curve passing by $p$ and the fixed points. Then the point $p$ is sent on the configuration given by these $2n$ points on this unique rational curve. Hence one sees that this is a birational map.

\begin{lemma}
The birational map $\varphi_{\Xi_{2n}}:\PP^{2n-3} \ra \cM_{0,2n}^{GIT}$ induced by $|\Xi_{2n}|$ resolves to a morphism $\tilde{\varphi}_{\Xi_{2n}}:\overline{\cM}_{0,2n} \ra \cM_{0,2n}^{GIT}$ that coincides with the contraction morphism 
$c_{2n}$ defined in Section \ref{moduli}.
\end{lemma}

\begin{proof}

Let us consider $\overline{\cM}_{0,2n}$ as a blow-up of $\PP^{2n-3}$ as described in Thm. \ref{bloua}. Then the base locus is contained in the locus that we recursively blow up hence the map $\varphi_{\Xi_{2n}}$ is resolved on $\overline{\cM}_{0,2n}$. Furthermore, on the open set $\cM_{0,2n}$ the morphism $\tilde{\varphi}_{\Xi_{2n}}$ and $c_{2n}$ coincide, hence they coincide everywhere and the following diagram commutes.

\begin{equation}\label{tria}
\xymatrix{ \overline{\cM}_{0,2n} \ar[dr]^{c_{2n}} \ar[d]_{b_{2n}} \\
\PP^{2n-3} \ar[r]^{\varphi_{\Xi_{2n}}} & \cM_{0,2n}^{GIT} 
  }
\end{equation}

\end{proof}

\begin{proposition}
The rational map $\mapd: \PP^{2n-2} \ra \cM^{GIT}_{0,2n}$ is the composed map $\varphi_{\Xi_{2n}}\circ \pi_{2n} \circ Cr_{2n-2}.$
\end{proposition}

\begin{proof}
Recall from Diagram \ref{diagram} that the composed map $\mapd\circ b_{2n+1}$ is equal to $c_{2n}\circ f_{2n+1}: \overline{\cM}_{0,2n+1} \ra \cM^{GIT}_{0,2n}.$ Now let us complete Diagram \ref{tria} with the forgetful maps in the following way.

\begin{equation*}
\xymatrix{\overline{\cM}_{0,2n+1} \ar[d]_{b_{2n+1}} \ar[rr]^{f_{2n+1}} & & \overline{\cM}_{0,2n} \ar[dr]^{c_{2n}} \ar[d]_{b_{2n}} \\
\PP^{2n-2} \ar[rr]^{\pi_{2n}\circ Cr_{2n-2}} & &\PP^{2n-3} \ar[r]^{\varphi_{\Xi_{2n}}} & \cM_{0,2n}^{GIT} 
  }
\end{equation*} 

This means that the rational maps $\mapd$ and $\varphi_{\Xi_{2n}}\circ \pi_{2n} \circ Cr_{2n-2}$ are resolved, via the Kapranov's blow-up, to the same morphism on $\overline{\cM}_{0,2n+1}$. This implies that they are the same map.


\end{proof}

\subsection{Examples of effective linear systems and contractions on $\overline{\cM}_{0,n}$}

We conclude these section by looking at our results in a slightly different framework, namely that of fibrations on  $\overline{\cM}_{0,n}$, ample and semi-ample divisors. We recall that a fibration on $\overline{\cM}_{0,n}$ is a proper and surjective morphism $f:\overline{\cM}_{0,n} \ra X$ with $f_*\cO_{\overline{\cM}_{0,n}}=\cO_X$, where X is a scheme with $dim(X)<dim(\overline{\cM}_{0,n})$. If $f$ is a fibration then $X$ is irreducible, normal and projective.
It is nowadays clear that, under some of the aspects just mentioned, the geometry of $\overline{\cM}_{0,n}$ is unexpecetdly more complicated than that of $\overline{\cM}_{g,n}$ for $g>0$. For instance the following theorems from \cite{GKM} are known to be true only for $g\geq 1$.

\begin{theorem}\label{vun}
For $g\geq 1$ any fibration of $\overline{\cM}_{g,n}$ to a projective variety factors through a projection to
some $\overline{\cM}_{g,i}$  with $i < n$.
\end{theorem}

\begin{theorem}\label{du}
Let $g\geq 1$ and let $f : \overline{\cM}_{g,n} \ra X$ be a birational morphism to a projective variety. Then the
exceptional locus of $f$ is contained in the boundary of $\overline{\cM}_{g,n}$. In particular X is again a compactification of $\cM_{g,n}$.
\end{theorem}

As I already stated, quite unexpectedly there are still no analogous very general theorems in genus zero. However some remarkable resultes have been obtained. For example in \cite{brumel} it is showed (Theorem 1.1) that, for $n\leq 7$, a dominant morphism $f: \overline{\cM}_{0,n} \ra X$ to a scheme factors as $f=g\circ h$, where $g$ is a generically finite morphism and $h$ belongs to certain classes of forgetful morphism.
We remark that for any even $n>2$, the fibrations defined by our maps $c_{2n}\circ f_{2n+1}$ are examples of one of the phenomenons described by Bruno and Mella in Thm. 1.1 of \cite{brumel}; namely, the fact that any fibration of $\overline{\cM}_{0,n}$ s.t. the generic fiber is a rational curve factors through the elementary forgetful map. However, in \cite{brumel} they also give an example of a fibration whose general fiber is an elliptic curve. On the other hand the contraction morphism $c_{2n}=\tilde{\varphi}_{\Xi_{2n}}$ is an example in genus 0 of a birational morphism like those of Theorem \ref{du}.

\bibliographystyle{amsalpha}
\bibliography{bibkaji} 

\def\cprime{$'$}
\providecommand{\bysame}{\leavevmode\hbox to3em{\hrulefill}\thinspace}
\providecommand{\MR}{\relax\ifhmode\unskip\space\fi MR }
\providecommand{\MRhref}[2]{%
  \href{http://www.ams.org/mathscinet-getitem?mr=#1}{#2}
}
\providecommand{\href}[2]{#2}
\begin{thebibliography}{HMSV09}

\bibitem[AB09]{albol}
Alberto Alzati and Michele Bolognesi, \emph{A structure theorem for
  $\mathcal{SU}(2)$ and the moduli of pointed genus zero curves}, 1 -- 17,
  (preprint, http://arxiv.org/abs/0903.5515).

\bibitem[AL02]{lanar}
D.~Avritzer and H.~Lange, \emph{The moduli spaces of hyperelliptic curves and
  binary forms}, Math. Z. \textbf{242} (2002), no.~4, 615--632.

\bibitem[BM09]{brumel}
Andrea Bruno and Massimiliano Mella, \emph{On fiber type morphisms of
  $\overline{\cM}_{0,n}$}, to appear (2009).

\bibitem[Cob30]{cobl2}
Arthur~B. Coble, \emph{A {G}eneralization of the {W}eddle {S}urface, of {I}ts
  {C}remona {G}roup, and of {I}ts {P}arametric {E}xpression in {T}erms of
  {H}yperelliptic {T}heta {F}unctions}, Amer. J. Math. \textbf{52} (1930),
  no.~3, 439--500.

\bibitem[Cob35]{cobwed}
A.~B. Coble, \emph{The geometry of the {W}eddle manifold {$W\sb p$}}, Bull.
  Amer. Math. Soc. \textbf{41} (1935), no.~4, 209--222.

\bibitem[Cob82]{coble}
Arthur~B. Coble, \emph{Algebraic geometry and theta functions}, American
  Mathematical Society Colloquium Publications, vol.~10, American Mathematical
  Society, Providence, R.I., 1982, Reprint of the 1929 edition.

\bibitem[DO88]{do:pstf}
Igor Dolgachev and David Ortland, \emph{Point sets in projective spaces and
  theta functions}, Asterisque, vol. 165, Soci\'{e}t\'{e} Math\'{e}matique de
  France, 1988.

\bibitem[FG03]{gabigib}
Gavril Farkas and Angela Gibney, \emph{The {M}ori cones of moduli spaces of
  pointed curves of small genus}, Trans. Amer. Math. Soc. \textbf{355} (2003),
  no.~3, 1183--1199.

\bibitem[GKM02]{GKM}
Angela Gibney, Sean Keel, and Ian Morrison, \emph{Towards the ample cone of
  {$\overline M\sb {g,n}$}}, J. Amer. Math. Soc. \textbf{15} (2002), no.~2,
  273--294 (electronic).

\bibitem[Har95]{har:ag}
J.~Harris, \emph{Algebraic geometry}, Graduate Texts in Mathematics, vol. 133,
  Springer-Verlag, New York, 1995, A first course.

\bibitem[Has03]{hassetpointed}
Brendan Hassett, \emph{Moduli spaces of weighted pointed stable curves}, Adv.
  Math. \textbf{173} (2003), no.~2, 316--352.

\bibitem[HMSV09]{vakilpointed}
Benjamin Howard, John Millson, Andrew Snowden, and Ravi Vakil, \emph{The
  equations for the moduli space of {$n$} points on the line}, Duke Math. J.
  \textbf{146} (2009), no.~2, 175--226.

\bibitem[Kap93a]{kapchowquot}
M.~M. Kapranov, \emph{Chow quotients of {G}rassmannians. {I}}, I. {M}.
  {G}el\cprime fand {S}eminar, Adv. Soviet Math., vol.~16, Amer. Math. Soc.,
  Providence, RI, 1993, pp.~29--110.

\bibitem[Kap93b]{kapra}
\bysame, \emph{Veronese curves and {G}rothendieck-{K}nudsen moduli space
  {$\overline M\sb {0,n}$}}, J. Algebraic Geom. \textbf{2} (1993), no.~2,
  239--262.

\bibitem[Kee92]{keel}
Sean Keel, \emph{Intersection theory of moduli space of stable {$n$}-pointed
  curves of genus zero}, Trans. Amer. Math. Soc. \textbf{330} (1992), no.~2,
  545--574.

\bibitem[KM96]{keelkernan}
Sean Keel and James McKernan, \emph{Contractible extremal rays on
  $\overline{\cM}_{0,n}$}, 1--24, preprint,
  http://arxiv.org/abs/alg-geom/9607009.

\bibitem[Knu83]{knudsenpointed}
Finn~F. Knudsen, \emph{The projectivity of the moduli space of stable curves.
  {II}. {T}he stacks {$M\sb{g,n}$}}, Math. Scand. \textbf{52} (1983), no.~2,
  161--199.

\bibitem[Kum00]{ku:ivb}
C.~Kumar, \emph{Invariant vector bundles of rank 2 on hyperelliptic curves},
  Michigan Math. J. \textbf{47} (2000), no.~3, 575--584.

\bibitem[Kum03]{kumar2}
Chanchal Kumar, \emph{Linear systems and quotients of projective space}, Bull.
  London Math. Soc. \textbf{35} (2003), no.~2, 152--160.

\bibitem[MFK82]{GIT}
David Mumford, John Fogarty, and Frances Kirwan, \emph{Geometric invariant
  theory}, Ergebnisse der mathematik und ihrer Grenzgebeite, vol.~34, Springer,
  Berlin, 1982.

\bibitem[Ver02]{verme}
Peter Vermeire, \emph{A counterexample to {F}ulton's conjecture on {$\overline
  M\sb {0,n}$}}, J. Algebra \textbf{248} (2002), no.~2, 780--784.

\end{thebibliography}

\end{document}